\documentclass[a4paper]{amsart}
\usepackage{verbatim}
\usepackage{graphicx}

\newcommand{\cal}{\mathcal}
\newcommand{\cN}{\mathcal{N}}

\newcommand{\dist}{{\rm dist}\,}
\newcommand{\rd}{{\mathbb R}^d}

\newcommand{\R}{{\mathbb R}}

\newcommand{\N}{{\mathbb N}}

\newcommand{\diam}{{\rm diam}\,}

\newcommand{\Reg}{{\rm Reg}}
\newcommand{\Ha}{{\cal H}}

\renewcommand{\searrow}{\to}
\newcommand{\var}{\textnormal{var}} 
\renewcommand{\S}{\mathbf{S}}
\newcommand{\bd}[1]{\partial #1} 
\newcommand{\eps}{\varepsilon}
\newcommand{\card}{\#}
\newcommand{\ind}[1]{\mathbf{1}_{#1}}
\newcommand{\nsubset}{\hspace{1ex}/\hspace{-2ex}\subset}
\newcommand{\mydot}{\,\cdot\,}

\newcommand{\wlim}[1]{\underset{#1}{\rm{wlim}}\,}
\newcommand{\esswlim}[1]{\underset{#1}{\rm{esswlim}}\,}
\newcommand{\esslim}[1]{\underset{#1}{\rm{esslim}}\,}

\newtheorem{thm}{Theorem}

\newtheorem{lem}[thm]{Lemma}
\newtheorem{cor}[thm]{Corollary}

\newtheorem{rem}[thm]{Remark}
\numberwithin{equation}{section} \numberwithin{thm}{section}
\begin{document}
\title{Fractal curvature measures of self-similar sets}
\author{Steffen Winter}
\address{Karlsruhe Institute of Technology, Department of Mathematics, 76133 Karlsruhe, Germany}
\author{Martina Z\"ahle}
\address{University of Jena, Mathematical Institute, 07737 Jena, Germany}

\date{\today}
\begin{abstract}
Fractal Lipschitz-Killing curvature measures $C^f_k(F,\mydot)$, $k= 0,\ldots,$ $d$, are determined for a large class of self-similar sets $F$ in $\rd$. They arise as weak limits of the appropriately rescaled classical Lipschitz-Killing curvature measures $C_k(F_\eps,\cdot)$ from geometric measure theory of parallel sets $F_\eps$ for small distances $\eps$. Due to self-similarity the limit measures appear to be constant multiples of the normalized Hausdorff measures on $F$, and the constants agree with the corresponding total fractal curvatures $C_k^f(F)$. This provides information on the 'second order' geometric fine structure of such fractals.
\end{abstract}
\maketitle

\section{Introduction}

Classical Lipschitz-Killing curvatures are well-known from convex geometry (Min\-kow\-ski's quermassintegrals for convex bodies), differential geometry
(integrals of mean curvatures for compact submanifolds with boundary) and geometric measure theory (Federer's curvature measures for sets of positive reach and additive extensions). They are intrinsically determined (cf. \cite{CMS84}, \cite{BK00}) and form a complete system of certain Euclidean invariants (see \cite{Za90}).\\
Fractal counterparts  have first been introduced in \cite{winter} for self-similar sets with polyconvex $\eps$-neighborhoods. Extensions to the random case with rather general parallel sets may be found in \cite{Za09}. Whereas in the latter paper only global curvatures are treated, the former contains also local refinements to fractal curvature measures. The aim of the present work is to introduce such measures for the general deterministic case from \cite{Za09}. As before we approximate the compact (fractal) sets $K$ by their $\eps$-neighborhoods
$$K_\eps:=\{x\in\rd: \dist(x,K)\le\eps\}\, .$$
We denote the {\it closure of the complement} of a compact set $K$ by $\widetilde{K}$.
A distance $\eps\ge 0$ is called {\it regular} for the set $K$ if $\widetilde{K_\eps}$ has positive reach in the sense of Federer \cite{F59} and the boundary $\partial K_\eps$ is a Lipschitz manifold. In view of Fu \cite{Fu85}, in space dimensions $d\le 3$ this is fulfilled for Lebesgue almost all $\eps$. (For general $d$, a sufficient condition for this property is that $\eps$ is a regular value of the distance function of $K$ in the sense of Morse theory, cf.~\cite{Fu85}.) For regular $\eps$ the {\it Lipschitz-Killing curvature measures} of order $k$ are determined by means of Federer's versions for sets of positive reach:
\begin{equation}
C_k(K_\eps,\mydot):=(-1)^{d-1-k}C_k(\widetilde{K_\eps},\mydot)\, ,~~k=0,\ldots,d-1\, ,
\end{equation}
where the surface area ($k=d-1$) is included and the volume measure $C_d(K_\eps,\mydot)$ is added for completeness. For more details and some background on classical singular curvature theory we refer to \cite{Za09} and \cite{winter}.\\
The {\it total curvatures} of $K_\eps$ are denoted by
\begin{equation}
C_k(K_\eps):=C_k(K_\eps,\rd)\, ,~~ k=0,\ldots ,d\, .
\end{equation}
We recall now the main properties of curvature measures required for our purposes:
By an associated Gauss-Bonnet theorem the {\it Gauss curvature}
$C_0(K_\eps)$ coincides with the {\it Euler-Poincar\'{e} characteristic} $\chi(K_\eps)$.\\
The curvature measures are {\it motion invariant}, i.e.,
\begin{equation}
C_k(g(K_\eps),g(\mydot))=C_k(K_\eps,\mydot)~~\mbox{for any Euclidean motion}~g\, ,
\end{equation}
the $k$-th measure is {\it homogeneous of degree} $k$, i.e.,
\begin{equation}
C_k(\lambda K_\eps,\lambda \mydot)=\lambda^k\, C_k(K_\eps,\mydot)\, ,~~\lambda>0\, ,
\end{equation}
and they are {\it locally determined}, i.e.,
\begin{equation}
C_k(K_\eps,\mydot\cap G)=C_k(K'_{\eps '},\mydot\cap G)
\end{equation}
for any open set $G\subset\rd$ such that $K_\eps\cap G=K'_{\eps '}\cap G$, where $K_\eps$ and $K'_{\eps '}$ are
both parallel sets such that the closures of their complements have positive reach.

Finally, for sufficiently large distances the parallel sets are always regular and the curvature measures may be estimated by those of a ball of almost the same size: For any compact set $K\subset\R^d$ and any $\eps>R>\sqrt{2}\,\diam K$ we have
\begin{align} \label{eq:R-big}
 C^\var_k(K_\eps)&\le c_k(K,R)\, \eps^k\,,
\end{align}
for some constant $c_k(K,R)$ independent of $\eps$, see \cite[Thm.~4.1]{Za09}). Here $C^\var_k(K_\eps,\mydot)$ denotes the total variation measure of $C_k(K_r,\mydot)$ and $C^\var_k(K_r):=C^\var_k(K_r,\R^d)$ its total mass.\\

In the present paper we consider self-similar sets $F$ in  $\rd$ with Hausdorff dimension $D$ satisfying the open set condition. Under some regularity condition on their $\eps$-neighborhoods we prove (Theorem~\ref{mainthm}) that the rescaled curvature measures $\eps^{D-k}C_k(F_\eps,\mydot)$ in a Ces\'{a}ro average weakly converge to some fractal limit measure $C_k^f(F,\cdot)$ as $\eps\rightarrow 0$. Due to the self-similarity the limit is a constant multiple of the normalized $D$-dimensional Hausdorff measure on $F$ and the constant agrees with the corresponding limit for the $k$th total curvatures from \cite{winter} and \cite{Za09}. If the contraction ratios of the similarities generating $F$ are non-arithmetic, the same result holds true for non-averaged essential limits as $\eps\rightarrow 0$. The constants, i.e. the total fractal curvatures $C_k^f(F)$, can be calculated in terms of integrals of the curvatures measures for small $\eps$-neighborhoods. Explicit numerical values for special examples may be found in \cite{winter}. Moreover, this paper also contains a discussion concerning the correct choice of the scaling exponent - in case the limit vanishes the exponent $D-k$ is not always appropriate.

\section{Main results}

For $N\in\N$ and $i=1,\ldots,N$, let $S_i:\R^d\to\R^d$ be a contracting similarity with contraction ratio $0<r_i<1$.
Let $F\subset\R^d$ be the \emph{self-similar set} generated by the function system $\{S_1,\ldots, S_N\}$. That is, $F$ is the unique nonempty, compact set invariant under the set mapping $\S (\mydot):=\bigcup_i S_i(\mydot)$, cf.~\cite{Hut81}. The set $F$ (or, more precisely, the system $\{S_1,\ldots,S_N\}$) is said to satisfy the \emph{open set
condition} (OSC) if there exists a non-empty, open and bounded
subset $O$ of $\R^d$
 such that
 $$
 \bigcup _i S_i O \subseteq O \quad \text{ and } \quad S_i O \cap
S_j O=\emptyset  \text{ for } i\neq j\,.
$$
The \emph{strong open set condition} (SOSC) holds for $F$ (or $\{S_1,\ldots,S_N\}$), if there exist a set $O$ as in the OSC which additionally satisfies $O\cap F\neq\emptyset$.
It is well known that in $\R^d$ OSC and SOSC are equivalent, cf.~\cite{schief}, i.e., for $F$ satisfying OSC, there exists always such a set $O$ with $O\cap F\neq\emptyset$.

The unique solution $s=D$ of the equation $\sum_{i=1}^N r_i^s=1$ is called the \emph{similarity dimension} of $F$.  It is well known that for self-similar sets $F$ satisfying OSC, $D$ coincides with Minkowski and Hausdorff dimension of $F$.
Further, a self-similar set $F$ is called \emph{arithmetic} (or \emph{lattice}), if there exists some number $h>0$ such
that $-\ln r_i \in h\mathbb{Z}$ for $i=1,\ldots, N$, i.e.\ if $\{-\ln r_1,\ldots,-\ln r_N\}$ generates a discrete subgroup of $\R$.
Otherwise $F$ is called \emph{non-arithmetic} (or \emph{non-lattice}).

Let $\Sigma^*:=\bigcup_{j=0}^\infty\{1,\ldots,N\}^j$ be set of all finite words over the alphabet $\{1,\ldots,N\}$ including the emtpy word. For $\omega=\omega_1\ldots \omega_n\in\Sigma^*$ we denote by $|\omega|$ the \emph{length of $\omega$} (i.e., $|\omega|=n$) and by $\omega|k:=\omega_1\ldots \omega_k$ the subword of the first $k\le n$ letters. We abbreviate $r_\omega:=r_{\omega_1}\ldots r_{\omega_n}$ and  $S_\omega:=S_{\omega_1}\circ\ldots\circ S_{\omega_n}$. Furthermore, let $r_{\min}:=\min\{r_i: 1\le i\le N\}$.

For a given self-similar set $F$, we
fix some constant $R=R(F)$ such that
\begin{equation}\label{eq:R}
R>\sqrt{2}\, \diam F
\end{equation}
(to be able to apply \eqref{eq:R-big}) and some open set $O=O(F)$ satisfying SOSC. Note that the choice of $R$ and the set $O$ are completely arbitrary. We fix both of them, because many of the sets and constants defined below depend on this choice.
For $0<\eps\le R$, let $\Sigma(\eps)$ be the family of all finite words $\omega=\omega_1\ldots \omega_n\in\Sigma^*$ such that
\begin{equation} \label{Sigma}
	Rr_\omega< \eps \le Rr_{\omega||\omega|-1},
\end{equation}
and let
\begin{equation} \label{Sigma_b}
	\Sigma_b(\eps):=\{\omega\in\Sigma(\eps): (S_\omega F)_\eps\cap (\S O)^c_\eps \neq \emptyset\}.
\end{equation}
The words $\omega $ in $\Sigma(\eps)$ describe those cylinder sets $S_\omega  F$ which are approximately of size $\eps$ and the words in $\Sigma_b(\eps)$ only those which are also $2\eps$-close to the boundary of the set $\S O$, the first iterate of the set $O$ under the set mapping $\S =\bigcup_{i=1}^N S_i$. Note that the family $\{S_\omega F: \omega\in\Sigma(\eps)\}$ is a covering of $F$ for each $\eps$, which is optimal in that none of the sets can be removed. It is an easy consequence of the equation $\sum_{i=1}^N r_i^D=1$ that, for each $\eps\in(0,R]$,
\begin{equation} \label{Sigma-eqn}
	\sum_{\omega\in\Sigma(\eps)} r_\omega^D=1\,.
\end{equation}
For convenience, we set $C_k(K_\eps,\mydot):=0$ whenever this measure is not defined otherwise, i.e., for non-regular $\eps$ and $k\le d-2$. (For $k=d-1$ and $k=d$, these measures can always be interpreted in terms of surface area and volume, respectively.)

The most general known global results on the limiting behaviour of curvature measures have been obtained in \cite{Za09} including random self-similar sets. We recall this result for the special case of deterministic sets.

\begin{thm} \cite[Theorem~2.3.8 and Corollary~2.3.9]{Za09} \label{thm:global}
Let $k\in\{0,1,\ldots,d\}$ and $F$ be a self-similar set in $\R^d$ satisfying OSC and the following two conditions.
\begin{enumerate}
\item[(i)] If $d\ge 4$ and $k\le d-2$, then almost all $\eps\in(0,R)$ are regular
for $F$. 
\item[(ii)] If $k\le d-2$, there is a constant $c_k$ such that for almost all $\eps\in(0,R)$ and all $\sigma\in \Sigma_b(\eps)$
\begin{equation}\label{eq:cond-ii}
C_k^\var\left(F_\eps, \bd (S_\sigma F)_\eps \cap \bd \bigcup_{\sigma'\in\Sigma(\eps)\setminus\{\sigma\}}(S_{\sigma'}F)_\eps\right)\le c_k\eps^k.
\end{equation}
\end{enumerate}
Set
\begin{equation} \label{Rk-def}
R_k(\eps) := C_k(F_\eps)-\sum_{i=1}^N \ind{(0,r_i]}(\eps) C_k((S_i F)_\eps),\quad \eps>0.
\end{equation}
Then
\begin{equation}\label{eq:global:lim}
C_k^f(F):=\lim_{\delta\searrow 0} \frac 1{|\ln\delta|}\int_\delta^1\eps^{D-k} C_k(F_\eps)\frac{d\eps}\eps =\frac 1{\eta} \int_0^R r^{D-k-1}R_k(r) dr,
\end{equation}
where $\eta=-\sum_{i=1}^N r_i^D \ln r_i$.
Moreover, if $F$ is non-arithmetic,
then
\begin{equation}\label{eq:global:esslim}
\esslim{\eps\searrow 0} \eps^{D-k} C_k(F_\eps)=C_k^f(F).
\end{equation}
\end{thm}

The numbers $C_k^f(F)$ are refered to as the \emph{fractal curvatures} of the set $F$.
For $k=d$, the limits in \eqref{eq:global:lim} and \eqref{eq:global:esslim} specialize to the average Minkowski content and the Minkowski content, respectively, and the result is due to Lapidus and Pomerance \cite{LapPo1}, Falconer \cite{fal93} (for $d=1$) and Gatzouras \cite{gatzouras} (for general $d$). The case $k=d-1$ has been treated in \cite{rw09}. In both cases the essential limits can be replaced by limits and  the limits are always positive. For the special case of polyconvex parallel sets, where the assumptions (i) and (ii) are not needed, see \cite{winter}.

Formula \eqref{eq:global:lim} in Theorem~\ref{thm:global} should in particular be understood to imply that the integral on the right hand side exists and thus the fractal curvatures are finite.
This follows indeed from the proof in \cite{Za09}. It is also directly seen from Theorem~\ref{unifbound} below.
To state it, let $\mathcal{N}\subset(0,R)$ be the set of values $\eps$ which are not regular for $F$. By condition (i) in Theorem~\ref{mainthm}, $\cN$ is a Lebsgue null set. Since critical values of the distance function of a compact set can not be larger than the diameter of the set, cf.~\eqref{eq:R-big}, and thus for $F$ not larger than $R$, it is clear that $\mathcal{N}$ contains all non-regular values of $F$. Furthermore, let $\mathcal{N}'\subset(0,R)$ be the null set for which the estimate \eqref{eq:cond-ii} in condition (ii) of Theorem~\ref{mainthm} does not hold. We set
\begin{equation}\label{eq:Nstar}
\mathcal{N}^*:=\bigcup_{\sigma \in\Sigma^*} r_\sigma (\mathcal{N}\cup\mathcal{N}')  \qquad \text{ and } \quad \Reg(F):=(0,\infty)\setminus \mathcal{N}^*\,.
\end{equation}
Observe that $\mathcal{N}^*\subset(0,R)$ and that it is a Lebesgue null set. Moreover, for each $\eps\in\Reg(F)$, not only $\eps$ but also $r_\sigma^{-1}\eps$ is regular for $F$ for each $\sigma\in\Sigma^*$. Consequently,
$\eps\in\Reg(F)$ is regular for $S_\sigma F$ for each $\sigma\in\Sigma^*$, which follows from the relation $(S_\sigma F)_\eps=S_\sigma(F_{\eps/r_\sigma})$. This implies in particular that the curvature measures $C_k(F_\eps,\mydot)$ and $C_k((S_{\sigma}F)_\eps,\mydot)$ are well defined for each $\eps\in\Reg(F)$ and each $\sigma\in\Sigma^*$.

Recall that $C_k^\var(F_\eps,\mydot)$ is the total variation measure of $C_k(F_\eps,\mydot)$ and $C_k^\var(F_\eps):=C_k^\var(F_\eps,\R^d)$. For $\eps\in\Reg(F)$, not only $C_k(F_\eps)$ but also $C_k^\var(F_\eps)$ is bounded as $\eps\to 0$ when rescaled with $\eps^{D-k}$.

\begin{thm}\label{unifbound}
Let $F$ be a self-similar set in $\R^d$ satisfying the hypotheses of Theorem~\ref{thm:global} and let $k\in\{0,\ldots,d-2\}$.
The expression $\eps^{D-k} C_k^\var(F_\eps)$ is uniformly bounded for $\eps\in\Reg(F)\cap(0,1]$, i.e.\
there is a constant $M$ such that for all $\eps\in\Reg(F)\cap(0,1]$, $\eps^{D-k} C_k^\var(F_\eps)\le M$.
\end{thm}

Note that for $k\in\{d-1,d\}$ the corresponding statement is an obvious consequence of Theorem~\ref{thm:global}, since in these cases the measure $C_k(F_\eps,\mydot)$ is positive and hence the total variation is just the measure itself. The proof of Theorem~\ref{unifbound} is given in Section~\ref{page:proof2}, see page~\pageref{page:proof2}.

Now we want to discuss our main result, the existence of (essential) weak limits of the suitably rescaled curvature measures of the parallel sets $F_\eps$ (as $\eps\searrow 0$) for self-similar sets. Let $k\in\{0,\ldots,d\}$.
Since weak convergence implies the convergence of the total masses of the measures, the measures $C_k(F_\eps,\mydot)$ have to be rescaled with the factor $\eps^{D-k}$ just as their total masses in Theorem~\ref{thm:global}. Therefore, 
for each $\eps\in\Reg(F)$, we define the \emph{$k$-th rescaled curvature measure} $\nu_{k,\eps}$ of $F_\eps$ by
\begin{equation}\label{resc-curv}
\nu_{k,\eps}(\mydot):=\eps^{D-k} C_k(F_\eps,\mydot)\,.
\end{equation}
In general,  the (essential) weak limit of these measures as $\eps\searrow 0$ need not exist. Often already the \emph{total masses}  $\nu_{k,\eps}(\R^d)=\eps^{D-k} C_k(F_\eps)$ fail to converge.
Therefore, and also to avoid taking essential limits, we define averaged versions $\overline{\nu}_{k,\eps}$ of the rescaled curvature measures $\nu_{k,\eps}$: For each $0<\eps<1$, let
\begin{equation}\label{avresc-curv}
\overline{\nu}_{k,\eps}(\mydot):=\frac{1}{|\ln\eps|}\int_\eps^1 \tilde{\eps}^{D-k} C_k(F_{\tilde{\eps}},\mydot)\frac{d\tilde{\eps}}{\tilde{\eps}}\,.
\end{equation}
Note that $\overline{\nu}_{k,\eps}$ is well defined, since the set $\mathcal{N}$ is assumed to be a null set.
For $k=d$ and $k=d-1$, the measures $\nu_{k,\eps}$ and $\overline{\nu}_{k,\eps}$ are positive, while for $k<d-1$, they are totally finite signed measures in general.
In the sequel we want to study the (essential) weak limits of the measures $\nu_{k,\eps}$ and $\overline{\nu}_{k,\eps}$ as $\eps\searrow 0$.
We will write $\wlim{\eps\searrow 0}\mu_\eps$ and $\esswlim{\eps \searrow 0}\mu_\eps$ for the \emph{weak limit} and the \emph{essential weak limit}, respectively, of a family of measures $\{\mu_\eps\}_{\eps\in(0,\eps_0)}$ as $\eps\searrow 0$. Here weak limit as $\eps\searrow 0$ means that the convergence takes place for any null sequence $\{\eps_n\}_{n\in\N}$ and essential weak limit means that there exist a set $\Lambda \subset(0,\eps_0)$ of Lebesgue measure zero such that the weak convergence takes place for any null sequence $\{\eps_n\}_{n\in \N}$ avoiding the set $\Lambda$. It will be clear from the proof that for the essential weak limits below the set $\Lambda$ to be avoided is the set $\mathcal{N}^*$ defined in \eqref{eq:Nstar}.\\
Our {\it main result on the existence and structure of fractal curvature measures} can now be formulated as follows.

\begin{thm}\label{mainthm}
Let $k\in\{0,1,\ldots,d\}$ and let $F$ be a self-similar set in $\R^d$ satisfying OSC and conditions (i) and (ii) of Theorem~\ref{thm:global}.
Then
\begin{equation}\label{eq:mainthm:wlim}
C_k^f(F,\mydot):=\wlim{\eps\searrow 0}\frac{1}{|\ln\eps|}\int_\eps^1 \tilde{\eps}^{D-k} C_k(F_{\tilde{\eps}},\mydot)\frac{d\tilde{\eps}}{\tilde{\eps}}=C_k^f(F)\mu_F,
\end{equation}
where
$\mu_F$ is the normalized $D$-dimensional Hausdorff measure on $F$. Moreover, if $F$ is non-arithmetic,
then
\begin{equation}\label{eq:mainthm:esswlim}
\esswlim{\eps\searrow 0}\eps^{D-k} C_k(F_\eps,\mydot)=C_k^f(F,\mydot)=C_k^f(F)\mu_F.
\end{equation}
\end{thm}


To verify the curvature bound condition (ii), it is necessary to look at all cylinder sets $S_\sigma F$ close to the ``boundary'' of $\S O$ and to determine for each of these sets $S_\sigma F$ the curvature (of $F_\eps$) in the intersection of the boundaries of $(S_\sigma F)_\eps$ and the union of all other cylinder sets sufficiently close $S_\sigma F$. These intersections are typically very small. Some concrete examples are discussed in \cite{w10}. They show in particular, that the results obtained in \cite{Za09} (for the total fractal curvatures) and here (for the fractal curvature measures) go clearly beyond the polyconvex setting in \cite{winter}.
\begin{rem} \label{rem:cond2}
It is not difficult to see that
\begin{equation} \label{eq:cond2}
(S_\sigma F)_\eps \cap \bigcup_{\sigma'\in\Sigma(\eps)\setminus\{\sigma\}}(S_{\sigma'}F)_\eps \cap \bd F_\eps = \bd(S_\sigma F)_\eps \cap \bd \bigcup_{\sigma'\in\Sigma(\eps)\setminus\{\sigma\}}(S_{\sigma'}F)_\eps \cap \bd F_\eps\, ,
\end{equation}
and therefore, since the curvature measures (of order $k\le d-1$) are concentrated on the boundary of $F_\eps$, condition $(ii)$ can equivalently be formulated with the boundary symbols in \eqref{eq:cond-ii} omitted. Several other equivalent formulations of this condition are presented in \cite{w10},
which illuminate the geometric meaning of condition $(ii)$ and simplify its verification.
\end{rem}

The remainder of the paper is organized as follows. In order to prepare the proof of the main results, we derive a number of estimates in the next section. The proof of Theorem~\ref{unifbound} is given in Section~\ref{page:proof2}, and the proof of Theorem~\ref{mainthm} in Section~\ref{sec:proof}.

\section{Some estimates} \label{sec:estimates}

 The most important result in this section is Lemma~\ref{key-lem}, while the main purpose of the other statements is to prove this lemma.
However, Lemma~\ref{curvBest} will be used again in Section~\ref{page:proof2}.
For $r>0$, let
\[
O(r):=\bigcup_{\sigma\in\Sigma(r)} S_\sigma O.
\]
\begin{lem} \label{key-lem}
Let $k\in\{0,\ldots,d\}$ and let $F$ be a self-similar set in $\rd$ satisfying the hypotheses of Theorem~\ref{mainthm}.
There exist positive constants $\rho, \gamma$ and, for each $r>0$, a positive constant $c=c(r)$ such that for all $\eps\in \Reg(F)$ and all $\delta$ with $0<\eps\le\delta\le \rho r$,
\[
C_k^\var(F_\eps, (O(r)^c)_\delta)\le c \eps^{k-D}\delta^{\gamma}.
\]
\end{lem}

Note that, for $k\in\{d-1,d\}$, this estimate holds for all $\eps$ not only those in $\Reg(F)$. The regularity is not required in these cases provided $C_{d-1}(F_\eps,\mydot)$ is interpreted as half the surface area of $F_\eps$, i.e. $C_{d-1}(F_\eps,\mydot)=\frac 12 \Ha^{d-1}(\bd F_\eps\cap\mydot)$. This is consistent with the definition given above. 

Lemma~\ref{key-lem} follows immediately by combining the Lemmas~\ref{curvBest} and~\ref{sigmaBest} below. For the proof of Lemma~\ref{curvBest} we require the following statement.
For $\eps\in(0,R)$ and $\sigma\in\Sigma(\eps)$, let
\begin{equation}\label{eq:A-sigma}
A^{\sigma,\eps}:=\bigcup_{\omega \in\Sigma(\eps)\setminus\{\sigma\}}(S_{\omega }F)_\eps.
\end{equation}

\begin{lem} \label{curv-est}
Under the hypotheses of Theorem~\ref{mainthm},
there is a constant $c>0$ such that, for all $\eps\in\Reg(F)\cap(0,R)$ (all $\eps\in(0,R)$, if $k\in\{d-1,d\}$) and all $\sigma\in\Sigma(\eps)$,
\begin{align}\label{eqn:curvB2}
C_k^\var(F_\eps,(S_\sigma F)_\eps\cap A^{\sigma,\eps})&\le c \eps^k\,.
\end{align}
\end{lem}
\begin{proof}
Let $m=|\sigma|$ and $\sigma=\sigma_1\ldots\sigma_m$ with $\sigma_i\in\{1,\ldots,N\}$. If $\sigma\in\Sigma_b(\eps)$, then, in view of equation \eqref{eq:cond2} in Remark~\ref{rem:cond2}, the assertion follows immediately from condition \eqref{eq:cond-ii}. If not, then $(S_\sigma F)_\eps\cap (\S O)^c_\eps=\emptyset$ and hence $(S_\sigma F)_\eps\subset (S_{\sigma_1} O)_{-\eps}$, where $A_{-\eps}:=((A^c)_\eps)^c$ denotes the (open) inner $\eps$-parallel set of a bounded set $A\subset\R^d$. In case $(S_\sigma F)_\eps\subset (S_\sigma O)_{-\eps}$,  there is nothing to prove, since the intersection  in \eqref{eqn:curvB2}  is  empty.  Otherwise let $1\le n<m$ be the index such that
\begin{equation}\label{eqn:curvB3}
(S_{\sigma} F)_\eps \subset (S_{\sigma_1\ldots\sigma_n} O)_{-\eps}
\end{equation}
but
\begin{equation}\label{eqn:curvB4}
(S_{\sigma} F)_\eps \nsubset
(S_{\sigma_1\ldots\sigma_{n+1}} O)_{-\eps}\,.
\end{equation}
With the notation $\sigma':=\sigma_1\ldots\sigma_n$ and $\sigma'':=\sigma_{n+1}\ldots\sigma_m$ (so that $\sigma=\sigma'\sigma''$), we infer from \eqref{eqn:curvB4} that
$
(S_{\sigma''}F)_{\eps/r_{\sigma'}}\nsubset (S_{\sigma_{n+1}}O)_{-\eps/r_{\sigma'}}
$
and thus
$
(S_{\sigma''}F)_{\eps/r_{\sigma'}}\cap (\S O)^c_{\eps/r_{\sigma'}}\neq\emptyset\,.
$
Hence $\sigma''\in\Sigma_b(\eps/r_{\sigma'})$. Since, by \eqref{eqn:curvB3}, we have $(S_{\sigma} F)_\eps\cap A^{\sigma,\eps} \subset (S_{\sigma'} O)_{-\eps}$, we can restrict the union in $A^{\sigma,\eps}$ to those $\omega =\omega _1\ldots,\omega _{m(\omega )}\in\Sigma(\eps)$ with $\omega _1\ldots \omega _n=\sigma'$. (A nonempty intersection of $(S_\omega F)_\eps\subset (S_\omega O)_\eps$ with the open set $(S_{\sigma'} O)_{-\eps}$ implies a nonempty intersection of $S_\omega O$ and $S_{\sigma'} O$ and thus $\omega _1\ldots \omega _n\neq\sigma'$ would contradict OSC.)
We infer that
\begin{align*}
(S_\sigma F)_\eps\,\cap A^{\sigma,\eps}
&= (S_{\sigma'\sigma''}F)_\eps\,\cap \bigcup_{\stackrel{\omega \in\Sigma(\eps)\setminus\{\sigma\}}{\omega _1\ldots \omega _n=\sigma'}}(S_\omega  F)_\eps \\
&= S_{\sigma'}\left((S_{\sigma''}F)_{\eps/r_{\sigma'}}\cap\bigcup_{\omega ''\in\Sigma(\eps/r_{\sigma'})\setminus\{\sigma''\}}(S_{\omega''} F)_{\eps/r_{\sigma'}}\right)\\
&=S_{\sigma'}\left( (S_{\sigma''} F)_\eps\,\cap A^{\sigma'',\eps/r_{\sigma'}} \right)
\end{align*}
and therefore,
by locality in the open set $(S_{\sigma'} O)_{-\eps}$\,,
\begin{align*}
C_k^\var(F_\eps, (S_\sigma F)_\eps\cap A^{\sigma,\eps})&=C_k^\var\left((S_{\sigma'}F)_\eps, (S_\sigma F)_\eps\cap A^{\sigma,\eps}\right)\\
&=C_k^\var\left(S_{\sigma'}(F_{\eps/r_{\sigma'}}), S_{\sigma'}\left( (S_{\sigma''} F)_{\eps/r_{\sigma'}}\,\cap A^{\sigma'',\eps/r_{\sigma'}}\right)\right)\\
&=r_{\sigma'}^k C_k^\var\left(F_{\eps/r_{\sigma'}},  (S_{\sigma''} F)_{\eps/r_{\sigma'}}\,\cap A^{\sigma'',\eps/r_{\sigma'}}\right)\\
&\le r_{\sigma'}^k \cdot c\, (\eps/r_{\sigma'})^k = c\, \eps^k\,.
\end{align*}
The last inequality is again due to condition \eqref{eq:cond-ii}, taking into account that $\sigma''\in\Sigma_b(\eps/r_{\sigma'})$ and that equation~\eqref{eq:cond2} in Remark~\ref{rem:cond2} holds. This completes the proof of Lemma~\ref{curv-est}.
\end{proof}

For a closed set $B\subseteq\R^d$ and $\eps\in(0,R)$, let
\begin{equation}\label{SigmaBdef}
\Omega(B,\eps):=\left\{\omega\in\Sigma(\eps): (S_\omega F)_\eps\cap B\neq\emptyset\right\}\,.
\end{equation}

\begin{lem} \label{curvBest} Under the hypotheses of Theorem~\ref{mainthm},
there is a positive constant $c$ such that, for all closed sets $B\subseteq\R^d$ and all $\eps\in\Reg(F)\cap(0,R)$ (all $\eps\in(0,R)$, if $k\in\{d-1,d\}$),
\[
C_k^\var(F_\eps, B)\le c\, \card \Omega(B,\eps) \eps^k\,.
\]
\end{lem}
\begin{proof} 
Let $\eps\in\Reg(F)\cap(0,R)$ or $k\in\{d-1,d\}$. We have
\begin{eqnarray*}
C_k^\var(F_\eps,B)
&=& C_k^\var\left(F_\eps, \bigcup_{\sigma\in\Sigma(\eps)} (S_\sigma F)_\eps\cap B\right)
\le C_k^\var\left(F_\eps, \bigcup_{\sigma\in\Omega(B,\eps)} (S_\sigma F)_\eps\right)\\
&\le& \sum_{\sigma\in\Omega(B,\eps)} C_k^\var(F_\eps,(S_\sigma F)_\eps)\,.
\end{eqnarray*}
For $k\in\{d,d-1\}$, $C_k^\var(F_\eps,(S_\sigma F)_\eps)$ is bounded from above by $C_k((S_\sigma F)_\eps)=r_\sigma^k C_k(F_{\eps/r_{\sigma}})\le (R^{-1}\eps)^k C_k(F_{\eps/r_{\sigma}})$. Since $\eps/r_\sigma\in(R, R/r_{\min}]$ for all $\sigma\in\Sigma(\eps)$, in case $k=d$, the monotonicity of the volume implies that $C_k^\var(F_\eps,(S_\sigma F)_\eps)\le c_k \eps^k$ with $c_d:=R^{-d} \lambda_d(F_{R/r_{\min}})$. For $k=d-1$, the corresponding estimate follows from the total boundedness of the surface area $\Ha^{d-1}(\bd F_r)$ for $r$ in the interval $[R, R/r_{\min}]$  (cf.~\cite[Corollary 4.2]{rw09}).

For $k\le d-2$ and $\eps\in\Reg(F)\cap(0,R)$, we use the sets $A^{\sigma,\eps}$ defined in \eqref{eq:A-sigma} to split the terms in the above sum as follows:
\begin{equation} \label{eqn:curvB}
C_k^\var(F_\eps,(S_\sigma F)_\eps)=C_k^\var(F_\eps,(S_\sigma F)_\eps\setminus A^{\sigma,\eps})+C_k^\var(F_\eps,(S_\sigma F)_\eps\cap A^{\sigma,\eps})\,.
\end{equation}
Since $\eps\in\Reg(F)$ is regular for $F$ and $S_\sigma F$, the locality property allows to replace $F_\eps$ by $(S_\sigma F)_\eps$ in the first term. 
Hence this term is bounded by
\[
C_k^\var((S_\sigma F)_\eps)=C_k^\var(S_\sigma (F_{\eps/r_\sigma}))=r^k_\sigma C_k^\var(F_{\eps/r_\sigma}).
\]
Since $\eps/r_\sigma> R$, we infer from \eqref{eq:R-big} the existence of a constant $c'$ (independent of $\sigma$ and $\eps$) such that $C_k^\var(F_{\eps/r_\sigma})\le c' (\eps/r_\sigma)^k$. Thus
$C_k^\var(F_\eps,(S_\sigma F)_\eps\setminus A^{\sigma,\eps})\le c' \eps^k$.

To the second term on the right of \eqref{eqn:curvB} we apply Lemma~\ref{curv-est}, which ensures that this term is bounded by $c''\eps^k$ for some constant $c''$ (independent of $\eps$ and $\sigma$).
Putting the bounds for the first and the second term in \eqref{eqn:curvB} back together and summing up over all $\sigma\in\Omega(B,\eps)$, the assertion follows immediately for the constant $c:=c'+c''$.
\end{proof}

It remains to show that for the choice $B=(O(r)^c)_\delta$, the cardinality of the sets $\Omega(B,\eps)$ is bounded as required. This follows easily from a similar result in \cite{winter}.
\begin{lem} \label{sigmaBest}
Let $F$ be a self-similar set in $\R^d$ satisfying OSC.
There exist positive constants $\rho, \gamma$ and, for each $r>0$, a positive constant $c=c(r)$ such that, for all $0<\eps\le\delta\le\rho r$,
\[\card{\Omega\left((O(r)^c)_\delta,\eps\right)} \le c \eps^{-D} \delta^{\gamma}.\]
\end{lem}
\begin{proof}
Fix $r>0$. By \cite[Lemma 5.4.1, p.47]{winter}, there are positive constants $\tilde c=\tilde c(r), \gamma$ and $\tilde\rho$ (with $\tilde\rho <1$) such that the cardinality of the set
\[
\Sigma((O(r)^c)_\delta,\eps):=\{\omega\in\Sigma(R\tilde\rho^{-1}\eps): (S_\omega F)_\eps\cap (O(r)^c)_\delta \neq\emptyset\}
\]
is bounded by $\tilde c\eps^{-D}\delta^\gamma$ for all $\eps\le\delta\le R^{-1}\tilde\rho r$. (Observe that the notation in \cite{winter} is slightly different to ours, due to a necessary extra constant $R$ in the definition of $\Sigma(\eps)$. More precisely, the sets $\Sigma(\eps)$ and $O(r)$ here coincide with $\Sigma(R^{-1}\eps)$ and $O(R^{-1}r)$ there and our $\tilde\rho$ is the $\rho$ in \cite{winter}. The sets $\Sigma(B,\eps)$ and the constant $\gamma$ are the same.) 
Therefore it suffices to show
that
\begin{equation}\label{Sigma-Omega_eqn}
\card{\Omega\left((O(r)^c)_\delta,\eps\right)}\le c'\card{\Sigma\left((O(r)^c)_\delta,\eps\right)}
\end{equation}
for some constant $c'$. Indeed this is true for any closed set $B\subseteq\rd$ instead of $(O(r)^c)_\delta$. We have
\begin{eqnarray*}
\card{\Omega\left(B,\eps\right)}&\le& \card\left\{\sigma \omega\in\Sigma(\eps): \sigma\in\Sigma(B,\eps), \omega\in\Sigma(\eps/r_\sigma)\right\}\\
&\le& \card \Sigma(B,\eps) \card\Sigma(\tilde\rho),
\end{eqnarray*}
where the last inequality is due to the fact that $\eps/r_\sigma>\tilde\rho$ for each $\sigma\in\Sigma(B,\eps)$ and that the cardinality of $\Sigma(r)$ is monotone decreasing. Since $\tilde\rho$ is fixed, $c':=\card\Sigma(\tilde\rho)$ is just a constant, proving \eqref{Sigma-Omega_eqn}. Hence the assertion of the lemma holds with $c:=\tilde c c'$, $\rho:= R^{-1}\tilde{\rho}$ and $\gamma$ as above.
\end{proof}

\begin{rem} \label{rem:rho} Note that the constants $\gamma$ and $\rho$ in Lemma~\ref{sigmaBest} (and Lemma~\ref{key-lem}) do not depend on $r$. From the proof of \cite[Lemma 5.4.1]{winter} it is clear that they just depend on the choice of the open set $O$, see also \cite[(5.1.8) and (5.1.9)]{winter} for the definition of $\gamma$ and $\tilde\rho$.
\end{rem}

\section{Proof of Theorem~\ref{unifbound}} \label{page:proof2}

Now we have all the ingredients to prove Theorem~\ref{unifbound}. In particular, we will make use of Lemma~\ref{key-lem} and Lemma~\ref{curvBest} above. 
In fact, we will only need the following simple consequence of Lemma~\ref{curvBest}.
\begin{cor}\label{cor:interval}
Let $0<a\le b\le R$. Under the hypotheses of Theorem~\ref{mainthm},
there is a positive constant $c'$ such that, for all $\eps\in\Reg(F)\cap [a,b]$ (all $\eps\in[a,b]$, if $k\in\{d-1,d\}$),
\[
C_k^\var(F_\eps)\le c'\,.
\]
\end{cor}
\begin{proof} Apply Lemma~\ref{curvBest}, with $B:=\R^d$ so that $\Sigma(B,\eps)=\Sigma(\eps)$ and observe that $c  \# \Sigma(\eps)\eps^k$
is bounded from above by $c':=c \# \Sigma(a) b^k$.
\end{proof}

\begin{proof}[Proof of Theorem~\ref{unifbound}]
Let $k\in\{0,\ldots,d-2\}$. First we will show that, for $\eps\in\Reg(F)$ and $0<r\le R$,
\begin{align} \label{eq:unifbound1}
C_k^\var(F_\eps)\le \sum_{\omega\in\Sigma(r)} C_k^\var((S_\omega F)_\eps) + C_k^\var(F_\eps,(O(r)^c)_\eps)\,.
\end{align}
For fixed $\eps$ and $r$ as above, let $U:=\bigcup_{v,\omega \in\Sigma(r)} (S_vF)_\eps\cap(S_\omega F)_\eps$ and $B^\omega :=(S_\omega F)_\eps\setminus U$ for $\omega \in\Sigma(r)$.
Then $F_\eps=U \cup \bigcup_{\omega \in\Sigma(r)} B^\omega $ and thus,
\[
C_k^\var(F_\eps)\le \sum_{\omega \in\Sigma(r)} C_k^\var(F_\eps, B^\omega ) + C_k^\var(F_\eps, U)\,.
\]
The set $A^\omega :=\big(\bigcup_{v\in\Sigma(r)\setminus\{\omega \}} (S_v F)_\eps\big)^c$ is open (the complement is a finite union of closed sets).
Moreover, $B^\omega \subseteq A^\omega $ and $F_\eps\cap A^\omega =(S_\omega F)_\eps\cap A^\omega $. Hence, by locality, we have
$C_k^\var(F_\eps, B^\omega )=C_k^\var((S_\omega  F)_\eps, B^\omega )\le C_k^\var((S_\omega  F)_\eps)$. It is easily seen that
$U\subset (O(r)^c)_\eps$ and so the inequality \eqref{eq:unifbound1} follows.

Now fix some $r<1$ and set $\tilde r:=Rr$. Applying Lemma~\ref{key-lem} and the equality $C_k^\var((S_\omega  F)_\eps)=r_\omega ^k C^\var_k(F_{\eps/r_\omega })$, we infer from \eqref{eq:unifbound1} that, for all $\eps\in\Reg(F)$ with $\eps<\rho \tilde r$,
\begin{align} \label{eq:unifbound2}
C_k^\var(F_\eps)\le \sum_{\omega \in\Sigma(\tilde r)} r_\omega^k C_k^\var(F_{\eps/r_\omega}) + c\, \eps^{k-D+\gamma}\,,
\end{align}
for some positive constants $c=c(\tilde r)$ and $\gamma$. To treat the interval $[\rho \tilde r,1]$, we infer from Corollary~\ref{cor:interval} that there exists a constant $c'=c'(\tilde r)$ such that, for all $\eps\in\Reg(F)\cap[\rho \tilde r,1]$,
$
C_k^\var(F_\eps)\le c'
$
and conclude that (by enlarging the constant $c$, if necessary) inequality \eqref{eq:unifbound2} holds, in fact,  for all $\eps\in\Reg(F)\cap(0,1]$.

Let $g:\Reg(F)\to\R$ be defined by $g(\eps):=\eps^{D-k}C_k^\var(F_\eps)$. We have to show that
$
\sup\{g(\eps): \eps\in\Reg(F)\cap(0,R]\}
$
is bounded by some positive constant $M$. 
By \eqref{eq:unifbound2}, the function $g$ satisfies, for $\eps\in\Reg(F)\cap(0,1]$,
\begin{equation} \label{ineq:g}
g(\eps)\le \sum_{\omega\in\Sigma(\tilde r)} r_\omega^D g(\eps/r_\omega)+c \eps^\gamma.
\end{equation}
For $n\in\N$ set $I_n:=\Reg(F)\cap(r^n,1]$ and $M_1:=\max\{\sup_{\eps\in I_1} g(\eps),c\}$. Note that $M_1<\infty$ is ensured by Corollary~\ref{cor:interval} above. We claim that for $n\in\N$,
\begin{equation}\label{ind1}
\sup_{\eps\in I_n}g(\eps) \le M_n:=M_1\sum_{j=0}^{n-1} (r^\gamma)^j,
\end{equation}
which we show by induction. For $n=1$, the statement is obvious. So assume that \eqref{ind1} holds for $n=k$. Then for $\eps\in I_k$, we have $g(\eps)\le M_k\le M_{k+1}$ and for $\eps\in I_{k+1}\setminus I_k$
we have $\eps/r_\omega \ge\eps/r\ge r^k$, i.e.\ $\eps/r_\omega\in I_k$ for all $\omega\in\Sigma(\tilde r)$. Hence, by \eqref{ineq:g} and \eqref{Sigma-eqn},
\[
g(\eps)\le \sum_{\omega\in\Sigma(\tilde r)} r_\omega^D g(\eps/r_\omega)+c\eps^\gamma\le \sum_{\omega\in\Sigma(\tilde r)} r_\omega^D M_k + M_1 r^{\gamma k}=M_{k+1},
\]
proving \eqref{ind1} for $n=k+1$ and hence for all $n\in\N$. Now observe that the sequence $(M_n)_{n\in\N}$ is bounded. Hence $g(\eps)$ is bounded in $(0,R]$, completing the proof of Theorem~\ref{unifbound}.
\end{proof}



Now we apply Lemma~\ref{key-lem} and Theorem~\ref{unifbound}, to derive some further technical estimates. We write $C_k^+(F_\eps,\mydot)$ and $C_k^-(F_\eps,\mydot)$ for
the positive and negative variation measure of $C_k(F_\eps,\mydot)$ (and, as before, $C_k^\var(F_\eps,\mydot)$ for the total variation).
\begin{lem}\label{curvbound}
Let $k\in\{0,1,\ldots,d\}$ and let $F$ be a self-similar set in $\R^d$ satisfying OSC and conditions (i) and (ii) of Theorem~\ref{thm:global}.
Let $\rho$ and $\gamma$ as in Lemma~\ref{key-lem} and let $\omega\in\Sigma^*$. Then there exists a positive constant $c=c(\omega)$ such that for all $\eps\in\Reg(F)$ and $\delta$ with $0<\eps\le\delta\le \rho r_\omega$
and for $\bullet\in\{+,-,\var\}$
\begin{equation}\label{upbound}
C_k^\bullet(F_\eps, (S_\omega O)_\delta)\le r_\omega ^k C_k^\bullet(F_{\eps r_\omega^{-1}}) + c \eps^{k-s} \delta^\gamma
\end{equation}
and
\begin{equation}\label{lowbound}
C_k^\bullet(F_\eps, S_\omega O)\ge r_\omega^k C_k^\bullet(F_{\eps r_\omega^{-1}}) - c \eps^{k-s} \delta^\gamma.
\end{equation}
\end{lem}
The proof is analogous to the one of \cite[Lemma~6.2.1, p.56]{winter} and therefore omitted. The main idea for the first estimate \eqref{upbound} is to
decompose $(S_\omega O)_\delta=(S_\omega O)_{-\delta}\cup(\bd S_\omega O)_\delta$ and apply the locality to the first set and Lemma~\ref{key-lem} to the second one. A similar argument works for \eqref{lowbound}.

In the sequel, we will write
$
\nu(f):=\int_{\R^d} f d\nu
$
for the integral of a function $f$ with respect to a (signed) measure $\nu$.
For $\omega \in\Sigma^*$ and $\delta>0$, let $f^\omega _\delta: \R^d\to[0,1]$ be a continuous function such that
\begin{equation}\label{fdeltadef}
f^\omega_\delta(x)=1 \quad \mbox{ for } x\in S_\omega O \quad \mbox{ and }\quad f^\omega_\delta(x)=0 \quad \mbox{ for } x \mbox{ outside } (S_\omega O)_\delta.
\end{equation}
For simplicity, assume that $f^\omega_\delta\le f^\omega_{\delta'}$ for all $\delta<\delta'$.
Obviously, $f^\omega_\delta$ has compact support and satisfies $\ind{S_\omega O}\le f^\omega_\delta \le \ind{(S_\omega O)_\delta}$. Moreover, as $\delta\searrow 0$, the functions $f^\omega_\delta$ converge (pointwise) to $\ind{S_\omega O}$, implying in particular the convergence of the integrals $\nu(f^\omega_\delta)\to \nu(\ind{S_\omega O})=\nu(S_\omega O)$ with respect to any signed Radon measure $\nu$.

Using Lemma~\ref{curvbound}, we derive some bounds for the integrals $\nu_{k,\eps}(f^\omega_\delta)$ and $\overline{\nu}_{k,\eps}(f^\omega_\delta)$. This will be an essential ingredient for the computation of the (essential) weak limits of the measures $\nu_{k,\eps}$ and $\overline{\nu}_{k,\eps}$.

\begin{lem}\label{curvbound2}
Let $k\in\{0,1,\ldots,d\}$ and let $F$ be a self-similar set in $\R^d$ satisfying OSC and conditions (i) and (ii) of Theorem~\ref{thm:global}.
Let $\omega\in\Sigma^*$. Let $c=c(\omega)$ be the constant of Lemma~\ref{curvbound} and $M$ the one in Theorem~\ref{unifbound}.
\begin{enumerate}
\item[(i)] For all $\eps\in\Reg(F)$ and $\delta$ such that $0<\eps\le\delta\le \rho r_\omega$, we have
\begin{equation*}
|\nu_{k,\eps}(f^\omega_\delta)-r_\omega^s \nu_{k,\eps r_\omega^{-1}}(\R^d)|\le 2 c \delta^\gamma.
\end{equation*}
\item[(ii)] For all $\eps$ and $\delta$ such that $0<\eps\le\delta\le \rho r_\omega$, we have
\begin{equation*}
 |\overline{\nu}_{k,\eps}(f^\omega_\delta)-r_\omega^s \overline{\nu}_{k,\eps r_\omega^{-1}}(\R^d)|\le 2 c
  \delta^\gamma + \frac{\ln\delta}{\ln\eps}2(c\delta^\gamma+M).
\end{equation*}
\end{enumerate}
\end{lem}
We omit the proof, since the arguments are analogous to those in the proofs of \cite[Lemmas~6.2.2 and 6.2.3]{winter}. The restriction to regular $\eps$  in (i) is required to ensure that Lemma~\ref{curvbound} can be applied (instead of \cite[Lemma~6.2.1]{winter} used in the proof of Lemma~6.2.2). Because of the averaging, in (ii) the regularity is not required, since by condition (i) in Theorem~\ref{mainthm} the set $\mathcal{N}^*$ is a null set.

\section{Proof of Theorem~\ref{mainthm}}\label{sec:proof}

First we recall the following uniqueness theorem for measures. It is well known for non-negative measures (see e.g.~\cite[p.51]{jacobs}) and easily generalized to signed measures (cf.~\cite[p.55]{winter}).

\begin{thm}\label{eindeut}
Let $\mu$ and $\nu$ be totally finite signed measures on the Borel $\sigma$-algebra $\mathfrak{B}^d$ of $\R^d$, and let ${\cal A}$ an intersection stable generator of $\mathfrak{B}^d$ such that $\mu(A)=\nu(A)$ for each set $A\in {\cal A}$. Then $\mu=\nu$.
\end{thm}
For a self-similar set $F$, let
\[
{\cal A}_F:=\{S_\omega O: \omega \in\Sigma^*\}\,\cup\, {\cal C}_F\, ,
\]
where
\[
{\cal C}_F:=\{ C\in\mathfrak{B}^d: \exists r>0 \mbox{ such that } C\subseteq O(r)^c\} \, .
\]
It is shown in \cite[Lemma~6.1.1]{winter} that the set family ${\cal A}_F$ is an intersection stable generator of $\mathfrak{B}^d$. Therefore, the above uniqueness theorem applies to ${\cal A}_F$. 

Now we are ready to give a proof of Theorem~\ref{mainthm}. We start with the non-lattice case and compute the essential weak limit of the $\nu_{k,\eps}$ in \eqref{eq:mainthm:esswlim}.

Let $F$ be a non-lattice self-similar set satisfying the hypotheses of Theorem~\ref{mainthm} and let $k\in\{0,\ldots,d\}$. According to condition (i) in Theorem~\ref{mainthm}, the set $\mathcal{N}\subset(0,\infty)$ of $\eps$ that are not regular for $F$ is a Lebesgue null set. Recall the definition of $
\mathcal{N}^*$ and $\Reg(F)$ from \eqref{eq:Nstar}.

First observe that the families $\{\nu^+_{k,\eps}: \eps\in(0,1)\setminus \mathcal{N}^*\}$ and $\{\nu^-_{k,\eps}: \eps\in(0,1)\setminus \mathcal{N}^*\}$ are tight. Indeed, by Theorem~\ref{unifbound}, the total masses of these measures are uniformly bounded and their support is contained in $F_1$. Hence, by Prokhorov's Theorem, every sequence in $\{\nu^+_{k,\eps}: \eps\in(0,1)\setminus \mathcal{N}^*\}$ (as well as in $\{\nu^-_{k,\eps}: \eps\in(0,1)\setminus \mathcal{N}^*\}$) has a weakly convergent subsequence. Hence starting from any null sequence $(\eps_n)_{n\in\N}$ we can always find a subsequence, for convenience again denoted by $(\eps_n)_{n\in\N}$, such that both sequences $(\nu^+_{k,\eps_n})_n$ and $(\nu^-_{k,\eps_n})_n$ converge weakly as $n\to\infty$. Let $\nu_k^+$ and $\nu_k^-$, respectively, denote the limit measure.
The weak convergence of the variation measures $\nu^+_{k,\eps_n}$ and $\nu^-_{k,\eps_n}$ implies the convergence of the (signed) measures $\nu_{k,\eps_n}=\nu^+_{k,\eps_n}-\nu^-_{k,\eps_n}$ and the limit measure is $\nu_k:=\nu_k^+ - \nu_k^-$.

A priori, the limit measure $\nu_k$ may depend on the chosen (sub)sequence $(\eps_n)$. However, we are going to show that the limit measure $\nu_k$ always coincides with the measure $\mu_k:=C_k^f(F)\mu_F$, independent of the sequence $(\eps_n)_n$.   The existence of the essential weak limit $\esswlim{\eps\searrow 0} \nu_{k,\eps}$ follows at once.

It remains to show that $\nu_k$ and $\mu_k$ coincide. By Theorem~\ref{eindeut}, it is enough to compare the values $\nu_k(A)$ and $\mu_k(A)$ for the sets $A$ of the family ${\cal A}_F$.
Recall that, for $\omega \in\Sigma^*$,
$$
\mu_k(S_\omega O)=C_k^f(F)\mu_F(S_\omega O)=C_k^f(F)\mu_F(S_\omega F)=C_k^f(F)r_\omega ^D.
$$
Moreover, for $C\in{\cal C}_F$, there is an $r>0$ such that $C\subseteq O(r)$ and hence
$$
\mu_k(C)=C_k^f(F)\mu_F(C)\le C_k^f(F)\mu_F(O(r)^c)=0.
$$
Therefore, it suffices to show that, for all $\omega \in{\Sigma^*}$,
\begin{equation}\label{nu-eqn}
\nu_k(S_\omega O)=C_k (F) r_\omega ^D,
\end{equation}
and for all $C\in{\cal C}_F$
\begin{equation}\label{nu-eqnC}
\nu_k(C)=0.
\end{equation}

\begin{proof}[Proof of \eqref{nu-eqn}]
Fix $\omega \in\Sigma^*$ and set $r:=r_\omega $. We approximate the measure of $S_\omega O$ by the integrals of the functions $f^\omega _\delta$ (defined in \eqref{fdeltadef}) and use Lemma~\ref{curvbound2}.
  Since the sequence $(\eps_n)$ avoids the set $\mathcal{N}^*$, by Lemma~\ref{curvbound2}(i), we have for all $n$ and $\delta$ such that $\eps_n\le\delta\le\rho r$
\begin{equation}
|\nu_{k,\eps_n}(f^\omega _\delta)-r_\omega ^D \nu_{k,\eps_n r_\omega ^{-1}}(\R^d)|\le 2 c \delta^\gamma.
\end{equation}
Keeping $\delta$ fixed and letting $n\to\infty$, the weak convergence implies $\nu_{k,\eps_n}(f^\omega _\delta)\to\nu_{k}(f^\omega _\delta)$, since $f^\omega _\delta$ is continuous. Moreover,
$\nu_{k,\eps_n r_\omega ^{-1}}(\R^d)=(\eps_n r_\omega ^{-1})^{D-k} C_k(F_{\eps_n r_\omega ^{-1}})\to C_k^f(F)$, by Theorem~\ref{thm:global}. Hence the above inequality yields
\begin{equation}
|\nu_{k}(f^\omega _\delta)-r_\omega ^D C_k (F)|\le 2 c \delta^\gamma
\end{equation}
for each $\delta\le\rho r$.
Letting now $\delta\to 0$, the integrals $\nu_{k}(f^\omega _\delta)$ converge to $\nu_{k}(\ind{S_\omega O})=\nu_k(S_\omega O)$, while the right hand side of the inequality vanishes.
Therefore, $|\nu_k(S_\omega O)-r_\omega ^D C_k (F)|\le 0$ which implies $\nu_k(S_\omega O)=r_\omega ^D C_k (F)$, as claimed in (\ref{nu-eqn}).
\end{proof}

\begin{proof}[Proof of \eqref{nu-eqnC}]
Fix $r>0$. It suffices to show $\nu_k^\pm(O(r)^c)=0$, since this immediately implies that $\nu_k(C)=\nu_k^+(C)-\nu_k^-(C)=0$ for all $C\subseteq O(r)^c$.
Similarly as before we approximate the indicator function of $O(r)^c$ by continuous functions.
For $\delta>0$, let $g_\delta :\R^d\to[0,1]$ be a continuous function such that
\begin{equation}\label{gdelta}
g_\delta(x)=1 \quad\mbox{ for }\quad x\in O(r)^c \quad\mbox{ and }\quad  g_\delta(x)=0 \quad\mbox{ for }\quad x\in (O(r))_{-\delta}.
\end{equation}
Since $g_\delta\le \ind{(O(r)^c)_\delta}$,
by Lemma~\ref{key-lem}, for all $\eps_n\le\delta\le\rho r$,
\begin{equation}\label{eq:g_delta}
\nu_{k,\eps_n}^\pm(g_\delta)\le c \delta^\gamma.
\end{equation}
Keeping $\delta$ fixed and letting $n\to\infty$, the weak convergence
implies that $\nu_{k,\eps_n}^{\pm}(g_\delta)\to \nu_k^{\pm}(g_\delta)$ while the right hand side remains unchanged.
Letting now $\delta\to 0$, the functions $g_\delta$ converge pointwise to $\ind{O(r)^c}$ and thus
$\nu_k^{\pm}(g_\delta)\to \nu_k^{\pm}(O(r)^c)$, while $c \delta^\gamma$ vanishes. Hence $\nu_k^{\pm}(O(r)^c)=0$, completing the proof of (\ref{nu-eqnC}).
\end{proof}

We have now completed the proof of \eqref{eq:mainthm:esswlim} in Theorem~\ref{mainthm}. It remains to provide a proof of \eqref{eq:mainthm:wlim}.
However, the arguments are now almost the same as in the proof of \eqref{eq:mainthm:esswlim}. Let $F$ be an arbitrary self-similar set satisfying the hypotheses in Theorem~\ref{mainthm}.
It is easily seen that the families $\{\overline{\nu}^+_{k,\eps}: \eps\in(0,1)\}$ and $\{\overline{\nu}^-_{k,\eps}: \eps\in(0,1)\}$ are tight.
Hence, by Prokhorov's Theorem, they are relatively compact. Let $\{\eps_n\}$ be a null sequence such that
\[
\wlim{n\to\infty}\overline{\nu}_{k,\eps_{n}}^+ = \overline{\nu}_k^+ \quad\mbox{ and }\quad \wlim{n\to\infty} \overline{\nu}_{k,\eps_{n}}^-= \overline{\nu}_k^-,
\]
for some limit measures $\overline{\nu}_k^+$ and $\overline{\nu}_k^-$ (which depend on the sequence $(\eps_n)$). Then $\wlim{n\to\infty}\overline{\nu}_{k,\eps_{n}}= \overline{\nu}_k:=\overline{\nu}_k^+ -\overline{\nu}_k^-$.
We have to show that $\overline{\nu}_k$ coincides with $\mu_k:=C_k^f(F)~\mu_F$, which implies the independence of $\overline{\nu}_k$ from the sequence $(\eps_n)$ and thus the convergence in \eqref{eq:mainthm:wlim}.
Employing again the set family ${\cal A}_F$ and Theorem~\ref{eindeut}, it remains to show that for all $\omega \in{\Sigma^*}$
\begin{equation}\label{avnu-eqn}
\overline{\nu}_k(S_\omega O)=C_k^f(F) r_\omega ^s,
\end{equation}
and for all $C\in{\cal C}_F$
\begin{equation}\label{avnu-eqnC}
\overline{\nu}_k(C)=0.
\end{equation}
The proofs of \eqref{avnu-eqn} and \eqref{avnu-eqnC} are completely analogous to the proofs of \eqref{nu-eqn} and \eqref{nu-eqnC} above.
For \eqref{avnu-eqn} use Lemma~\ref{curvbound2} (ii), and for \eqref{avnu-eqnC}, note that
for all $\eps\le\delta\le\rho r$,
\begin{equation*}
\overline{\nu}_{k,\eps}^{\pm}(g_\delta)\le c \delta^\gamma + \frac{\ln\delta}{\ln\eps}(c\delta^\gamma+M).
\end{equation*}
This is easily derived from \eqref{eq:g_delta} (which holds for all $\eps$ not just $\eps_n$) and Theorem~\ref{unifbound}.

\end{document}